\documentclass[a4paper]{article}%
\usepackage{amssymb,latexsym}
\usepackage{amsfonts}
\usepackage{amsmath}
\usepackage[colorlinks=true, pdfstartview=FitV, linkcolor=blue,
citecolor=blue, urlcolor=blue]{hyperref}
\usepackage{color}
\usepackage{amssymb}
\usepackage{graphicx}%
\providecommand{\U}[1]{\protect\rule{.1in}{.1in}}
\newtheorem{theorem}{Theorem}

\newtheorem{definition}[theorem]{Definition}

\newtheorem{lemma}[theorem]{Lemma}

\newtheorem{remark}[theorem]{Remark}

\newenvironment{proof}[1][Proof]{\noindent\textbf{#1.} }{\ \rule{0.5em}{0.5em}}
\begin{document}

\title{On generalized Eisenstein series and Ramanujan's formula for periodic zeta-functions}
\author{M. Cihat Da\u{g}l\i\ and M\"{u}m\"{u}n Can\\Department of Mathematics, Akdeniz University, \\07058-Antalya, Turkey\\e-mails: mcihatdagli@akdeniz.edu.tr, mcan@akdeniz.edu.tr}
\date{}
\maketitle

\begin{abstract}
In this paper, transformation formulas for a large class of Eisenstein series
defined by
\[
G(z,s;A_{\alpha},B_{\beta};r_{1},r_{2})=\sum\limits_{m,n=-\infty}^{\infty
}\ \hspace{-0.19in}^{^{\prime}}\frac{f(\alpha m)f^{\ast}(\beta n)}%
{((m+r_{1})z+n+r_{2})^{s}},\text{ }\operatorname{Re}(s)>2,\text{
}\operatorname{Im}(z)>0
\]
are investigated for $s=1-r,$ $r\in\mathbb{N}$. Here $\left\{  f(n)\right\}  $
and $\left\{  f^{\ast}(n)\right\}  ,$ $-\infty<n<\infty$ are sequences of
complex numbers with period $k>0,$ and $A_{\alpha}=\left\{  f(\alpha
n)\right\}  $ and $B_{\beta}=\left\{  f^{\ast}(\beta n)\right\}  ,$
$\alpha,\beta\in\mathbb{Z}.$ Appearing in the transformation formulas are
generalizations of Dedekind sums involving the periodic Bernoulli function.
Reciprocity law is proved for periodic Apostol-Dedekind sum outside of the
context of the transformation formulas. Furthermore, transformation formulas
are presented for $G(z,s;A_{\alpha},I;r_{1},r_{2})$ and $G(z,s;I,A_{\alpha
};r_{1},r_{2}),$ where $I=\left\{  1\right\}  .$ As an application of these
formulas, analogues of Ramanujan's formula for periodic zeta-functions are derived.



\end{abstract}

\section{Introduction}

For integers $c$ and $d$ with $c>0$, the classical Dedekind sum $s(d,c),$
arising in the transformation formula of the Dedekind eta-function
$\eta\left(  z\right)  ,$ is defined by
\[
s(d,c)=\sum\limits_{n(\operatorname{mod}c)}\left(  \left(  \frac{n}{c}\right)
\right)  \left(  \left(  \frac{dn}{c}\right)  \right)
\]
where the \textit{sawtooth} function is defined by%
\[
\left(  \left(  x\right)  \right)  =%
\begin{cases}
x-\left[  x\right]  -1/2, & \text{if }x\in\mathbb{R}\backslash\mathbb{Z}%
\text{,}\\
0, & \text{if }x\in\mathbb{Z}%
\end{cases}
\]
with $[x]$\ the floor function. One of the most important properties of
Dedekind sums is the reciprocity formula \textbf{\ }%
\begin{equation}
s(d,c)+s(c,d)=-\frac{1}{4}+\frac{1}{12}\left(  \frac{d}{c}+\frac{c}{d}%
+\frac{1}{dc}\right)  \label{35}%
\end{equation}
whenever $c$ and $d$ are coprime positive integers. For several proofs of
(\ref{35}) and generalizations for Dedekind sums, for example, see
\cite{14,12,2,11,7,15,20,5,18,ham,lim,9,nos,13,sek,16}.

There are several other functions such as Eisenstein series, which possess
transformation formula similar to $\log\eta(z)$. Lewittes \cite{6} has
discovered a method of obtaining transformation formulas for certain
generalized Eisenstein series which developed by Berndt \cite{12}. Berndt's
transformation formulas contain transformation formulas previously established
by Apostol, C. Meyer, Dieter, Schoeneberg and others. By this way, Berndt
obtain transformation formula for several types of Eisenstein series involving
exponential function and primitive characters \cite{11,2,7}. Following Berndt,
Meyer \cite{9}, Sekine \cite{sek1} and Lim \cite{lim} give\textbf{
}transformation formulas for aforementioned classes of functions. Arising in
the transformation formulas are various types of Dedekind sums, all of which
satisfy reciprocity theorems.

From the principal theorem of \cite{11}, Berndt \cite{22} and Goldberg
\cite{gold}\ derived a number of transformation formulas for the logarithms of
the classical theta-functions in which Dedekind-like sums or Hardy-Berndt sums
appear. A character generalization of certain Hardy-Berndt sums have been
achieved by Can and Kurt \cite{ck} from the main theorem of \cite{2}.

One of the most interesting corollaries of \cite[Theorem 2]{12} is Ramanujan's
formula for $\zeta(2n+1),$ $n\geq1$. It is shown in \cite{8} that this and
also Euler's formula for $\zeta(2n),$ $n\geq1,$ are both consequences of
\cite[Theorem 2]{12}. Moreover, the character analogue of Ramanujan's formula
for Dirichlet $L$-function is due to Katayama \cite{kat} and is proved by
Berndt \cite{7} via transformation formulas. In \cite{brad}, Bradley has
achieved periodic analogues of Ramanujan's formula by employing the partial
fraction expansion of the hyperbolic cotangent. Katayama's and Berndt's
formulas are immediate consequences of Bradley's theorems.

Recently, the authors \cite{dc} have derived transformation formulas for a
very large class of Eisenstein series defined by
\begin{equation}
G(z,s;A_{\alpha},B_{\beta};r_{1},r_{2})=\sum\limits_{m,n=-\infty}^{\infty
}\ \hspace{-0.19in}^{^{\prime}}\frac{f(\alpha m)f^{\ast}(\beta n)}%
{((m+r_{1})z+n+r_{2})^{s}},\text{ }\operatorname{Re}(s)>2,\text{
}\operatorname{Im}(z)>0 \label{53}%
\end{equation}
where $\left\{  f(n)\right\}  $ and $\left\{  f^{\ast}(n)\right\}  ,$
$-\infty<n<\infty$ are sequences of complex numbers with period $k>0,$ and
$A_{\alpha}=\left\{  f(\alpha n)\right\}  $ and $B_{\beta}=\left\{  f^{\ast
}(\beta n)\right\}  ,$ $\alpha,\beta\in\mathbb{Z}.$ In (\ref{53}), the dash
$\prime$ means that the possible pair $m=-r_{1},n=-r_{2}$ is excluded from the
summation. These transformation formulas are extensions of the principal
theorems of Berndt \cite{12,2,7}. Letting $s=0$ in these transformation
formulas, generalizations of Dedekind sums (called periodic Dedekind sums)
involving the periodic Bernoulli function appear. It is shown that these sums
still obey reciprocity laws. In addition, relations between various infinite
series and evaluations of several infinite series are deduced from the
transformation formulas. Also, periodic Dedekind sums for some special values
of $A_{\alpha}$ and $B_{\beta}$ are illustrated.

In the present paper, we wish to investigate transformation formulas for
positive integer $r$ and $s=1-r$ in Theorem \ref{mainteo} (see below).
Appearing in these transformation formulas are generalizations of Dedekind
sums as%
\[
\sum\limits_{n=1}^{ck}f^{\ast}(bn)P_{m}\left(  \frac{n}{ck}\right)
P_{r+1-m}(dn/c,A_{c})
\]
where $P_{r}(x,A_{c})$ is the periodic Bernoulli function (see (\ref{pr**})).
In Section 3, we prove two reciprocity formulas. The first one is for the
function involving periodic Dedekind sums and the second is for periodic
Apostol-Dedekind sum.\textbf{ }In Section 4, we present transformation
formulas which yield analogues of Ramanujan's formula for periodic
zeta-functions. Bradley's formulas are consequences of one of our theorems and
a very special cases of an infinite class of similar formulas.

\section{Notation and preliminaries}

Let $\mathbb{H=}\left\{  z:\operatorname{Im}(z)>0\right\}  $ denote the upper
half-plane and $\mathbb{K=}\{z:\operatorname{Re}(z)>-d/c$, $\operatorname{Im}%
(z)>0\}$ denote the upper quarter--plane. We use the modular transformation
$Vz=V\left(  z\right)  =\left(  az+b\right)  /\left(  cz+d\right)  $ where
$a,$ $b,$ $c$ and $d$ are integers with $ad-bc=1$ and $c>0$. We use the
notation $\left\{  x\right\}  $ for the fractional part of $x,$ and
$\lambda_{x}$ for the characteristic function of integers. We put
$e(z)=e^{2\pi iz}$ and unless otherwise stated, we use the branch of the
argument defined by $-\pi\leq\arg z<\pi$.

Let $\left\{  f(n)\right\}  =A,$ $-\infty<n<\infty$ be sequence of complex
numbers with period $k>0.$ For $\left\vert t\right\vert <2\pi/k,$ the periodic
Bernoulli numbers and polynomials are defined by means of the generating
functions \cite{10}%
\begin{equation}
\sum\limits_{n=0}^{k-1}\frac{tf(n)e^{nt}}{e^{kt}-1}=\sum\limits_{j=0}^{\infty
}\frac{B_{j}(A)}{j!}t^{j} \label{b}%
\end{equation}
and
\begin{equation}
\sum\limits_{n=0}^{k-1}\frac{tf(-n)e^{(n+x)t}}{e^{kt}-1}=\sum\limits_{j=0}%
^{\infty}\frac{B_{j}(x,A)}{j!}t^{j}. \label{bp}%
\end{equation}
Note that, when $A=I=\left\{  1\right\}  ,$(\ref{b}) and (\ref{bp}) reduce to
ordinary Bernoulli numbers and polynomials, defined by the generating
functions \cite{1}%
\begin{align}
\frac{t}{e^{t}-1}  &  =\sum\limits_{n=0}^{\infty}B_{n}\frac{t^{n}}%
{n!},\ |t|<2\pi,\nonumber\\
\frac{te^{xt}}{e^{t}-1}  &  =\sum\limits_{n=0}^{\infty}B_{n}(x)\frac{t^{n}%
}{n!},\ |t|<2\pi, \label{8}%
\end{align}
respectively. Notice that\textbf{ }$B_{0}(x)=1,$ $B_{1}=-1/2,$ $B_{1}(1)=1/2,$
$B_{2n+1}=B_{2n-1}\left(  1/2\right)  =0,$\ $n\geq1,$\ and $B_{1}(x)=x-1/2.$

Throughout this paper, the $n-$th Bernoulli function will be denoted by
$P_{n}\left(  x\right)  $ and is defined by
\[
n!P_{n}\left(  x\right)  =B_{n}\left(  x-\left[  x\right]  \right)  .
\]
In particular $P_{1}\left(  x\right)  =x-\left[  x\right]  -1/2.$ These
functions satisfy\textbf{ }the Raabe or multiplication formula for all real
$x$
\begin{equation}
P_{n}(x)=k^{n-1}\sum\limits_{m=0}^{k-1}P_{n}\left(  \frac{m+x}{k}\right)
\label{26}%
\end{equation}
and the reflection identity $P_{n}(-x)=\left(  -1\right)  ^{n}P_{n}(x)$ except
$n=1$ and $x\in\mathbb{Z}$, in that case
\[
P_{1}(-x)=P_{1}(x)=P_{1}(0)=-1/2.
\]

The periodic Bernoulli functions $P_{r}(x,A),$ are functions with period $k$,
may be defined by \cite{10}
\[
P_{0}(x,A)=B_{0}(A)=\frac{1}{k}\sum\limits_{m=0}^{k-1}f(m)
\]
and%
\begin{equation}
P_{r}(x,A)=k^{r-1}\sum\limits_{v=0}^{k-1}f(-v)P_{r}\left(  \frac{v+x}%
{k}\right)  ,\text{ }r\geq1 \label{10-1}%
\end{equation}
for all real $x$. Also, the periodic Bernoulli function $P_{r}(x,A_{c})$ is
defined in \cite{dc} by%
\begin{equation}
P_{r}(x,A_{c})=k^{r-1}\sum\limits_{v=0}^{k-1}f(-cv)P_{r}\left(  \frac{v+x}%
{k}\right)  . \label{pr**}%
\end{equation}

Define the sequence $\widehat{A}=\left\{  \widehat{f}(n)\right\}  $ by
\begin{equation}
\widehat{f}(n)=\frac{1}{k}\sum\limits_{j=0}^{k-1}f\left(  j\right)  e\left(
-nj/k\right)  \label{47}%
\end{equation}
for $-n<\infty<n.$ These are the finite Fourier series coefficients of
$\left\{  f(n)\right\}  .$ Clearly $\widehat{A}$ also has period $k.$ Note
that (\ref{47}) holds if and only if%
\begin{equation}
f(n)=\sum\limits_{j=0}^{k-1}\widehat{f}\left(  j\right)  e\left(  nj/k\right)
,\text{ \ }-n<\infty<n.\text{\ } \label{2}%
\end{equation}
We need the following Lemma and Theorems given in \cite{dc}.

\begin{lemma}
\label{lem}Let $z\in\mathbb{H}$, $\operatorname{Re}(s)>2$ and $\beta\beta
^{-1}\equiv1(\operatorname{mod}k).$ Then,
\begin{align}
&  \Gamma(s)G(z,s;A_{\alpha},B_{\beta};r_{1},r_{2})\nonumber\\
&  =\left(  -2\pi i/k\right)  ^{s}k\left(  A\left(  z,s;A_{\alpha}%
,\widehat{B}_{-\beta^{-1}};r_{1}{\huge ,}r_{2}\right)  +e(s/2)A\left(
z,s;A_{-\alpha},\widehat{B}_{\beta^{-1}};-r_{1}{\LARGE ,}-r_{2}\right)
\right) \nonumber\\
&  \quad+\lambda_{r_{1}}f(-\alpha r_{1})\Gamma(s)\left(  L(s;B_{\beta}%
;r_{2})+e(s/2)L(s;B_{-\beta};-r_{2})\right)  \label{11}%
\end{align}
where%
\begin{equation}
A(z,s;A_{\alpha},A_{\beta};r_{1}{\huge ,}r_{2})=\sum\limits_{m>-r_{1}}f(\alpha
m)\sum\limits_{n=1}^{\infty}f(\beta n)e\left(  n\frac{(m+r_{1})z+r_{2}}%
{k}\right)  n^{s-1} \label{25}%
\end{equation}
and%
\begin{equation}
L(s;A_{\beta};\theta)=\sum\limits_{n>-\theta}f(n\beta)(n+\theta)^{-s},\text{
for }\operatorname{Re}\left(  s\right)  >1\text{ and }\theta\text{ real.}
\label{12}%
\end{equation}

\end{lemma}

Note that since $L(s;B_{\beta};r_{2})$ can be analytically continued to the
entire $s$--plane with the possible exception $s=1$ and since $A\left(
z,s;A_{\alpha},B_{\beta};r_{1}{\huge ,}r_{2}\right)  $ is entire function of
$s,$ $G(z,s;A_{\alpha},B_{\beta};r_{1},r_{2})$ can be analytically continued
to the entire $s$--plane with the possible exception $s=1$.

\begin{theorem}
\label{mainteo} Define $R_{1}=ar_{1}+cr_{2}${\huge \ }and $R_{2}=br_{1}%
+dr_{2},$ in which $r_{1}$ and $r_{2}$ are arbitrary real numbers. Let
$\rho=\rho\left(  R_{1},R_{2},c,d\right)  =\left\{  R_{2}\right\}  c-\left\{
R_{1}\right\}  d.$ Suppose first that $a\equiv d\equiv0\left(
\operatorname{mod}k\right)  .$ Then for $z\in\mathbb{K}$ and all $s,$
\begin{align}
&  (cz+d)^{-s}\Gamma(s)G(Vz,s;A,B;r_{1},r_{2})\label{13-a}\\
&  =\Gamma(s)G(z,s;B_{-b},A_{-c};R_{1},R_{2})-2i\Gamma(s)\sin(\pi
s)L(s;A_{c};-R_{2})f^{\ast}(bR_{1})\lambda_{R_{1}}\nonumber\\
&  \quad+e(-s/2)\sum\limits_{j=1}^{c}\sum\limits_{\mu=0}^{k-1}\sum
\limits_{v=0}^{k-1}f(c(\left[  R_{2}+d(j-\left\{  R_{1}\right\}  )/c\right]
-v))f^{\ast}(b(\mu c+j+\left[  R_{1}\right]  ))\nonumber\\
&  \qquad\qquad\times I(z,s,c,d,r_{1},r_{2})\nonumber
\end{align}
where $L(s;A_{c};R_{2})$ is given by (\ref{12}) and
\begin{align}
&  I(z,s,c,d,r_{1},r_{2})\nonumber\\
&  \ =\int\limits_{C}u^{s-1}\frac{\exp(-\left(  (c\mu+j-\left\{
R_{1}\right\}  )/ck\right)  (cz+d)ku)}{\exp(-ku(cz+d))-1}\frac{\exp((\left(
v+\left\{  (dj+\rho)/c\right\}  \right)  /k)ku)}{\exp(ku)-1}du. \label{int}%
\end{align}
Here, we choose the branch of $u^{s}$ with $0<\arg u<2\pi$. Also, $C$ is a
loop beginning at $+\infty$, proceeding in the upper half-plane, encircling
the origin in the positive direction so that $u=0$ is the only zero of
$(\exp(-ku(cz+d))-1)(\exp\left(  ku\right)  -1)$ lying "inside" the loop, and
then returning to $+\infty$ in the lower half-plane.

Secondly, if $b\equiv c\equiv0\left(  \operatorname{mod}k\right)  ,$ then for
$z\in\mathbb{K}$ and all $s,$%
\begin{align}
&  (cz+d)^{-s}\Gamma(s)G(Vz,s;A,B;r_{1},r_{2})\label{13}\\
&  =\Gamma(s)G(z,s;A_{d},B_{a};R_{1},R_{2})-2i\Gamma(s)\sin(\pi s)L(s;B_{-a}%
;-R_{2})f(-dR_{1})\lambda_{R_{1}}\nonumber\\
&  \quad+e(-s/2)\sum\limits_{j=1}^{c}\sum\limits_{\mu=0}^{k-1}\sum
\limits_{v=0}^{k-1}f^{\ast}(-a(\left[  R_{2}+d(j-\left\{  R_{1}\right\}
)/c\right]  -v+d\mu))f(-d(j+\left[  R_{1}\right]  ))I(z,s,c,d,r_{1}%
,r_{2}).\nonumber
\end{align}

\end{theorem}

\begin{theorem}
Under the conditions of Theorem \ref{mainteo}, for $a\equiv d\equiv0\left(
\operatorname{mod}k\right)  $ we have%
\begin{align}
&  (cz+d)^{-s}\Gamma(s)G(Vz,s;B_{-\beta},A_{-\alpha};r_{1},r_{2})\label{34}\\
&  =\Gamma(s)G(z,s;A_{\alpha b},B_{\beta c};R_{1},R_{2})-2i\Gamma(s)\sin(\pi
s)f(-\alpha bR_{1})L(s;B_{-\beta c};-R_{2})\nonumber\\
&  +e(-s/2)\sum\limits_{j=1}^{c}\sum\limits_{\mu=0}^{k-1}\sum\limits_{v=0}%
^{k-1}f(-\alpha b(\mu c+j+\left[  R_{1}\right]  ))f^{\ast}(-\beta c(\left[
R_{2}+d(j-\left\{  R_{1}\right\}  )/c\right]  -v))\nonumber\\
&  \qquad\times I(z,s,c,d,r_{1},r_{2}),\nonumber
\end{align}
where $I(z,s,c,d,r_{1},r_{2})$ is given by (\ref{int}).
\end{theorem}

For simplicity, the function $G(z,s;A_{\alpha},B_{\beta};0,0)$\ will be
denoted by $G(z,s;A_{\alpha},B_{\beta})$ and $G(z,s;A_{1},$ $B_{1};r_{1}%
,r_{2})=G\left(  z,s;A,B;r_{1},r_{2}\right)  $.

\section{The periodic analogue of Dedekind sum}

Theorem \ref{mainteo} can be simplified when $s=1-r$ is an integer for
$r\geq1$. In \cite{dc}, we investigate Theorem \ref{mainteo} for the case
$r=1$ and$\ r_{1},r_{2}\in\mathbb{R}$ in detail. Therefore, in this section we
assume that $r>1$ and $r_{1}=r_{2}=0.$ In this case, by the residue theorem,
$I(z,1-r,c,d,0,0)$ becomes with the aid of (\ref{8})%

\begin{align}
&  I(z,1-r,c,d,0,0)\label{30}\\
&  =\frac{2\pi ik^{r-1}}{\left(  r+1\right)  !}\sum\limits_{m=0}^{r+1}%
\binom{r+1}{m}\left(  -(cz+d)\right)  ^{m-1}B_{m}\left(  \frac{c\mu+j}%
{ck}\right)  B_{r+1-m}\left(  \frac{v+\left\{  dj/c\right\}  }{k}\right)
.\nonumber
\end{align}
Let $a\equiv d\equiv0\left(  \operatorname{mod}k\right)  .$ Substituting
(\ref{30}) in (\ref{13-a}) gives
\begin{align}
&  \lim_{s\rightarrow1-r}\left(  (cz+d)^{-s}\Gamma(s)G(Vz,s;A,B)-\Gamma
(s)G(z,s;B_{-b},A_{-c})\right) \nonumber\\
&  =-2if^{\ast}(0)\lim_{s\rightarrow1-r}\Gamma(s)\sin(\pi s)L(s,A_{c}%
;0)\nonumber\\
&  \quad+\left(  -1\right)  ^{r-1}\frac{2\pi ik^{r-1}}{\left(  r+1\right)
!}\sum\limits_{m=0}^{r+1}\binom{r+1}{m}\left(  -(cz+d)\right)  ^{m-1}%
\nonumber\\
&  \quad\times\sum\limits_{j=1}^{c}\sum\limits_{\mu=0}^{k-1}\sum
\limits_{v=0}^{k-1}f(c(\left[  dj/c\right]  -v))f^{\ast}(b(\mu c+j))B_{m}%
\left(  \frac{c\mu+j}{ck}\right)  B_{r+1-m}\left(  \frac{v+\left\{
dj/c\right\}  }{k}\right)  . \label{31}%
\end{align}
We must evaluate the triple sum in (\ref{31}). We first note that the triple
sum is invariant by replacing $B_{m}\left(  \frac{c\mu+j}{ck}\right)  $ by
$m!P_{m}\left(  \frac{c\mu+j}{ck}\right)  $ for $1\leq j\leq c-1$ and
$B_{r+1-m}\left(  \frac{v+\left\{  dj/c\right\}  }{k}\right)  $ by $\left(
r+1-m\right)  !P_{r+1-m}\left(  \frac{v+\left\{  dj/c\right\}  }{k}\right)  $.
Also replacing $v-\left[  dj/c\right]  $ by $v$ yields
\begin{align*}
&  \sum\limits_{j=1}^{c}\sum\limits_{\mu=0}^{k-1}\sum\limits_{v=0}%
^{k-1}f(c(\left[  dj/c\right]  -v))f^{\ast}(b(\mu c+j))B_{m}\left(  \frac
{c\mu+j}{ck}\right)  B_{r+1-m}\left(  \frac{v+\left\{  dj/c\right\}  }%
{k}\right) \\
&  =m!\left(  r+1-m\right)  !\sum\limits_{j=1}^{c-1}\sum\limits_{\mu=0}%
^{k-1}f^{\ast}(b(\mu c+j))P_{m}\left(  \frac{c\mu+j}{ck}\right)
\sum\limits_{v=0}^{k-1}f(-cv))P_{r+1-m}\left(  \frac{v+dj/c}{k}\right) \\
&  \quad+\left(  r+1-m\right)  !\sum\limits_{\mu=0}^{k-1}f^{\ast}%
(bc(\mu+1))B_{m}\left(  \frac{\mu+1}{k}\right)  \sum\limits_{v=0}%
^{k-1}f(-cv)P_{r+1-m}\left(  \frac{v}{k}\right)  .
\end{align*}
Here adding and subtracting\textbf{ }the term\textbf{ }$j=c$ and
writing\textbf{ }$\mu c+j=n,$ the right-hand side\textbf{ }becomes
\begin{align*}
&  m!\left(  r+1-m\right)  !k^{m-r}\sum\limits_{n=1}^{ck}f^{\ast}%
(bn)P_{m}\left(  \frac{n}{ck}\right)  P_{r+1-m}\left(  \frac{dn}{c}%
,A_{c}\right) \\
&  \quad+\left(  r+1-m\right)  !k^{m-r}f^{\ast}(0)P_{r+1-m}\left(
0,A_{c}\right)  \left(  B_{m}\left(  1\right)  -m!P_{m}\left(  0\right)
\right) \\
&  =m!\left(  r+1-m\right)  !k^{m-r}\sum\limits_{n=1}^{ck}f^{\ast}%
(bn)P_{m}\left(  \frac{n}{ck}\right)  P_{r+1-m}\left(  \frac{dn}{c}%
,A_{c}\right) \\
&  \quad+r!k^{1-r}f^{\ast}(0)P_{r}\left(  0,A_{c}\right)  \text{.}%
\end{align*}
Then, we can give the following definition.

\begin{definition}
Let $ad-bc=1$ and $a\equiv d\equiv0$ $\left(  \operatorname{mod}k\right)  .$
The periodic Dedekind sum $s_{r+1-m,m}\left(  d,c;B_{b};A_{c}\right)  $ is
defined by%
\[
s_{r+1-m,m}\left(  d,c;B_{b};A_{c}\right)  =\sum\limits_{n=1}^{ck}f^{\ast
}(bn)P_{m}\left(  \frac{n}{ck}\right)  P_{r+1-m}(dn/c,A_{c}).
\]

\end{definition}

On the other hand, using the Euler's reflection formula $\Gamma(s)\Gamma
(1-s)=\pi/\sin\left(  \pi s\right)  $ and the fact \cite[Corollary 6.4 for
$a=0$]{10}, we can write
\begin{equation}
\lim_{s\rightarrow1-r}\Gamma(s)\sin(\pi s)L(s,A_{c};0)=\pi\left(  -1\right)
^{r-1}P_{r}(0,A_{c}). \label{55}%
\end{equation}
\textbf{ }

Thus, we have
\begin{align*}
&  \lim_{s\rightarrow1-r}\Gamma(s)\left(  (cz+d)^{-s}G(Vz,s;A,B)-\frac{{}}{{}%
}G(z,s;B_{-b},A_{-c})\right) \\
&  =\left(  -1\right)  ^{r-1}2\pi i\sum\limits_{m=0}^{r+1}k^{m-1}\left(
-(cz+d)\right)  ^{m-1}s_{r+1-m,m}\left(  d,c;B_{b};A_{c}\right)  .
\end{align*}

For $b\equiv c\equiv0\left(  \operatorname{mod}k\right)  ,$ by similar
arguments, one can hold%
\begin{align*}
&  \lim_{s\rightarrow1-r}\Gamma(s)\left(  (cz+d)^{-s}G(Vz,s;A,B)-\frac{{}}{{}%
}G(z,s;A_{d},B_{a})\right) \\
&  =\left(  -1\right)  ^{r-1}2\pi i\sum\limits_{m=0}^{r+1}k^{m-1}\left(
-(cz+d)\right)  ^{m-1}\sum\limits_{n=1}^{ck}f(-dn)P_{m}\left(  \frac{n}%
{ck}\right)  P_{r+1-m}(dn/c,B_{-a}).
\end{align*}

\begin{definition}
Let $ad-bc=1$ and $b\equiv c\equiv0$ $\left(  \operatorname{mod}k\right)
.$\ The periodic Dedekind sum $s_{r+1-m,m}\left(  d,c;A_{-d};B_{-a}\right)  $
is defined by%
\[
s_{r+1-m,m}\left(  d,c;A_{-d};B_{-a}\right)  =\sum\limits_{n=1}^{ck}%
f(-dn)P_{m}\left(  \frac{n}{ck}\right)  P_{r+1-m}(dn/c,B_{-a}).
\]

\end{definition}

We summarize the results obtained above in the next theorem.

\begin{theorem}
Let $r>1$ be an integer and $z\in\mathbb{H}$. If\textbf{\ }$a\equiv
d\equiv0\left(  \operatorname{mod}k\right)  ,$ then
\begin{align}
&  \lim_{s\rightarrow1-r}\Gamma(s)\left(  (cz+d)^{-s}G(Vz,s;A,B)-\frac{{}}{{}%
}G(z,s;B_{-b},A_{-c})\right) \nonumber\\
&  =2\pi i\left(  -1\right)  ^{r-1}\sum\limits_{m=0}^{r+1}k^{m-1}\left(
-(cz+d)\right)  ^{m-1}s_{r+1-m,m}\left(  d,c;B_{b};A_{c}\right)  . \label{56}%
\end{align}
If $b\equiv c\equiv0\left(  \operatorname{mod}k\right)  ,$ then
\begin{align}
&  \lim_{s\rightarrow1-r}\Gamma(s)\left(  (cz+d)^{-s}G(Vz,s;A,B)-\frac{{}}{{}%
}G(z,s;A_{d},B_{a})\right) \nonumber\\
&  =2\pi i\left(  -1\right)  ^{r-1}\sum\limits_{m=0}^{r+1}k^{m-1}\left(
-(cz+d)\right)  ^{m-1}s_{r+1-m,m}\left(  d,c;A_{-d};B_{-a}\right)  .
\label{57}%
\end{align}

\end{theorem}

\subsection{Reciprocity Theorems}

In this section, we prove two reciprocity theorems. The first one can be
viewed as the reciprocity formula for the function $F\left(  d,c,z;r;A_{c}%
;B_{b}\right)  $ given by
\begin{equation}
F\left(  d,c,z;r;A_{c};B_{b}\right)  =\sum\limits_{m=0}^{r+1}k^{m-1}\left(
-(dz+c)\right)  ^{m-1}s_{r+1-m,m}\left(  c,d;A_{c};B_{b}\right)  . \label{60}%
\end{equation}
The second one is the reciprocity formula for periodic Apostol--Dedekind sum
$s_{r}\left(  d,c;A_{\alpha};B_{\beta}\right)  ,$ defined by%
\begin{equation}
s_{r}\left(  d,c;A_{\alpha};B_{\beta}\right)  =\sum\limits_{n=1}^{ck}f(\alpha
n)P_{1}\left(  \frac{n}{ck}\right)  P_{r}(dn/c,B_{\beta}). \label{a-ds}%
\end{equation}

\begin{theorem}
Let $ad-bc=1$ with $d,c>0$ and $r>1$ be an integer. For\textbf{\ }$a\equiv
d\equiv0\left(  \operatorname{mod}k\right)  $ and $z\in\mathbb{C}-\left\{
0,c/d\right\}  ,$ we have%
\begin{align*}
&  F\left(  d,-c,z;r;A_{c};B_{-b}\right)  -z^{r-1}F\left(  c,d,-\frac{1}%
{z};r;B_{b};A_{c}\right) \\
&  =\sum\limits_{m=0}^{r+1}\left(  -z\right)  ^{m-1}P_{m}\left(
0,A_{-c}\right)  P_{r+1-m}\left(  0,B_{-b}\right)  ,
\end{align*}
where $F\left(  d,c,z;r;A_{c};B_{b}\right)  $ is defined by (\ref{60}).
\end{theorem}

\begin{proof}
Replacing $z$ by $-1/z$ in (\ref{56}) gives
\begin{align}
&  \lim_{s\rightarrow1-r}\Gamma(s)z^{s}\left(  \frac{1}{(dz-c)^{s}%
}G(Tz,s;A,B)-z^{-s}G(-\frac{1}{z},s;B_{-b},A_{-c})\right) \label{56*}\\
&  =\left(  -1\right)  ^{r-1}2\pi i\sum\limits_{m=0}^{r+1}k^{m-1}\left(
-\left(  \frac{dz-c}{z}\right)  \right)  ^{m-1}s_{r+1-m,m}\left(
d,c;B_{b};A_{c}\right) \nonumber
\end{align}
and replacing $Vz$ by $Tz=\left(  bz-a\right)  /\left(  dz-c\right)  =V\left(
-1/z\right)  $ in (\ref{57}) gives%
\begin{align}
&  \lim_{s\rightarrow1-r}\Gamma(s)\left(  \frac{1}{(dz-c)^{s}}%
G(Tz,s;A,B)-G(z,s;A_{-c},B_{b})\right) \label{57*}\\
&  =\left(  -1\right)  ^{r-1}2\pi i\sum\limits_{m=0}^{r+1}k^{m-1}\left(
-(dz-c)\right)  ^{m-1}s_{r+1-m,m}\left(  -c,d;A_{c};B_{-b}\right)  .\nonumber
\end{align}
Since%
\begin{align*}
&  \lim_{s\rightarrow1-r}z^{s}\Gamma(s)\left(  \frac{1}{\left(  dz-c\right)
^{s}}G\left(  Tz,s;A,B\right)  -G(z,s;A_{-c},B_{b})\right) \\
&  =\lim_{s\rightarrow1-r}z^{s}\Gamma(s)\left(  \frac{1}{\left(  dz-c\right)
^{s}}G(V\left(  -1/z\right)  ,s;A,B)-\frac{1}{z^{s}}G\left(  -1/z,s;B_{-b}%
,A_{-c}\right)  \right) \\
&  \quad+\lim_{s\rightarrow1-r}z^{s}\Gamma(s)\left(  \frac{1}{z^{s}}G\left(
-1/z,s;B_{-b},A_{-c}\right)  -G(z,s;A_{-c},B_{b})\right)  ,
\end{align*}
it is sufficient to evaluate the following limit in order to prove the
reciprocity formula%
\[
\lim_{s\rightarrow1-r}z^{s}\Gamma(s)\left(  \frac{1}{z^{s}}G\left(
-1/z,s;B_{-b},A_{-c}\right)  -G(z,s;A_{-c},B_{b})\right)  .
\]
For this, taking $Vz=-1/z$ in (\ref{34}) and using (\ref{30}) we have
\begin{align}
&  \lim_{s\rightarrow1-r}\Gamma(s)\left(  z^{-s}G(-1/z,s;B_{-b},A_{-c}%
)-G(z,s;A_{-c},B_{b})\right) \nonumber\\
&  =-2\pi if(0)\left(  -1\right)  ^{r-1}P_{r}(0,B_{-b})\nonumber\\
&  \quad+\frac{2\pi i\left(  -k\right)  ^{r-1}}{\left(  r+1\right)  !}%
\sum\limits_{m=0}^{r+1}\binom{r+1}{m}\left(  -z\right)  ^{m-1}\sum
\limits_{\mu=0}^{k-1}f(c(\mu+1))B_{m}\left(  \frac{\mu+1}{k}\right)
\sum\limits_{v=0}^{k-1}f^{\ast}(bv)B_{r+1-m}\left(  \frac{v}{k}\right)  .
\label{14}%
\end{align}
It is seen from (\ref{pr**}) that
\begin{align*}
&  \sum\limits_{v=0}^{k-1}f^{\ast}(bv)B_{r+1-m}\left(  \frac{v}{k}\right)
\left(  \sum\limits_{\mu=1}^{k-1}f\left(  c\mu\right)  B_{m}\left(  \frac{\mu
}{k}\right)  +f\left(  0\right)  B_{m}\left(  1\right)  \right) \\
&  =(r+1-m)!k^{m-r}P_{r+1-m}\left(  0,B_{-b}\right)  \left(  m!k^{1-m}%
P_{m}\left(  0,A_{-c}\right)  +%
\begin{cases}
f(0), & \text{if }m=1\text{;}\\
0, & \text{otherwise,}%
\end{cases}
\right)  .
\end{align*}
Substituting these in (\ref{14}) and using (\ref{55}) yield%
\begin{align}
&  \lim_{s\rightarrow1-r}\Gamma(s)\left(  \frac{1}{z^{s}}G(-1/z,s;B_{-b}%
,A_{-c})-G(z,s;A_{-c},B_{b})\right) \label{34*}\\
&  =2\pi i\left(  -1\right)  ^{r-1}\sum\limits_{m=0}^{r+1}\left(  -z\right)
^{m-1}P_{m}\left(  0,A_{-c}\right)  P_{r+1-m}\left(  0,B_{-b}\right)
.\nonumber
\end{align}

Thus, combining (\ref{34*}), (\ref{56*}) and (\ref{57*}) gives
\begin{align}
&  F\left(  d,-c,z;r;A_{c};B_{-b}\right)  -z^{r-1}F\left(  c,d,-\frac{1}%
{z};r;B_{b};A_{c}\right) \nonumber\\
&  =\sum\limits_{m=0}^{r+1}\left(  -z\right)  ^{m-1}P_{m}\left(
0,A_{-c}\right)  P_{r+1-m}\left(  0,B_{-b}\right)  , \label{37}%
\end{align}
for $z\in\mathbb{H}$. By analytic continuation,\ (\ref{37}) is valid for
$z\in\mathbb{C}-\left\{  0,c/d\right\}  .$
\end{proof}

Before stating and proving a reciprocity formula for the periodic
Apostol--Dedekind sum $s_{r}\left(  d,c;A_{\alpha};B_{\beta}\right)  $ given
by (\ref{a-ds}), we discuss this sum.\textbf{ }

Assume that $\alpha\equiv0\left(  \operatorname{mod}k\right)  .$ Since
$A_{\alpha}=\left\{  f\left(  0\right)  \right\}  =f\left(  0\right)  I,$ it
follows that%
\begin{align*}
s_{r}\left(  d,c;A_{\alpha};B_{\beta}\right)   &  =f\left(  0\right)
\sum\limits_{n=1}^{ck}P_{1}\left(  \frac{n}{ck}\right)  P_{r}(dn/c,B_{\beta
}),\\
s_{r}\left(  d,c;B_{\beta};A_{\alpha}\right)   &  =f\left(  0\right)
\sum\limits_{n=1}^{ck}f^{\ast}(\beta n)P_{1}\left(  \frac{n}{ck}\right)
P_{r}\left(  \frac{dn}{c}\right)  ,
\end{align*}
which are periodic extensions of Berndt's character Dedekind sums
$S_{2}\left(  d,c;\chi\right)  $ and $S_{1}\left(  d,c;\chi\right)  $ defined
in \cite[Eqs. (6.2) and (6.1)]{7}, respectively.

If $\alpha\equiv\beta\equiv0\left(  \operatorname{mod}k\right)  ,$ then the
sum $s_{r}\left(  d,c;A_{\alpha};B_{\beta}\right)  $ is equal to $f\left(
0\right)  f^{\ast}(0)s_{r}\left(  d,c\right)  ,$ where $s_{r}\left(
d,c\right)  $ is the Apostol--Dedekind sum
\[
s_{r}\left(  d,c\right)  =\sum\limits_{j=0}^{c-1}P_{1}\left(  \frac{j}%
{c}\right)  P_{r}\left(  \frac{dj}{c}\right)  .
\]

In general, writing $v+jk,$ $1\leq v\leq k,$ $0\leq j<c$ in place of $n$ in
the definition of $s_{r}\left(  d,c;A_{\alpha};B_{\beta}\right)  $ and using
(\ref{pr**}), we have%
\begin{align*}
s_{r}\left(  d,c;A_{\alpha};B_{\beta}\right)   &  =\sum\limits_{n=1}%
^{ck}f(\alpha n)P_{1}\left(  \frac{n}{ck}\right)  P_{r}(dn/c,B_{\beta})\\
&  =k^{r-1}\sum\limits_{v=1}^{k}\sum\limits_{j=0}^{c-1}f(\alpha v)P_{1}\left(
\frac{j+\frac{v}{k}}{c}\right)  \sum\limits_{\mu=1}^{k}f^{\ast}(-\beta
\mu)P_{r}\left(  \frac{d\left(  j+\frac{v}{k}\right)  }{c}+\frac{\mu}%
{k}\right)  \\
&  =k^{r-1}\sum\limits_{v=1}^{k}\sum\limits_{\mu=1}^{k}f(\alpha v)f^{\ast
}(-\beta\mu)\sum\limits_{j=0}^{c-1}P_{1}\left(  \frac{j+\frac{v}{k}}%
{c}\right)  P_{r}\left(  \frac{d\left(  j+\frac{v}{k}\right)  }{c}+\frac{\mu
}{k}\right)  .
\end{align*}
Observe that the sum over $j$ is the generalized Dedekind sums due to Carlitz
\cite{20} (or\textbf{ }Takacs \cite{16}) given by
\[
s_{r}\left(  d,c|x,y\right)  =\sum\limits_{j=0}^{c-1}\overline{B}_{1}\left(
\frac{j+y}{c}\right)  \overline{B}_{r}\left(  \frac{d\left(  j+y\right)  }%
{c}+x\right)
\]
where $\overline{B}_{r}\left(  x\right)  =r!P_{r}\left(  x\right)  .$ Then%
\begin{equation}
s_{r}\left(  d,c;A_{\alpha};B_{\beta}\right)  =\frac{k^{r-1}}{r!}%
\sum\limits_{v=1}^{k}\sum\limits_{\mu=1}^{k}f(\alpha v)f^{\ast}(-\beta
\mu)s_{r}\left(  d,c|\frac{\mu}{k},\frac{v}{k}\right)  .\label{1}%
\end{equation}
This shows that we may achieve a reciprocity formula for $s_{r}\left(
d,c;A_{\alpha};B_{\beta}\right)  $ with the help of the following reciprocity
law for $s_{r}\left(  d,c|x,y\right)  :$ Let $c$ and $d$ be positive coprime
integers. For integers $r\geq0$ and for real numbers $x$ and $y$, the
reciprocity formula holds \cite{20,16}%
\begin{align}
&  \left(  r+1\right)  \left[  dc^{r}s_{r}\left(  d,c|x,y\right)  +cd^{r}%
s_{r}\left(  c,d|y,x\right)  \right]  \nonumber\\
&  =\sum\limits_{j=0}^{r+1}\binom{r+1}{j}c^{j}d^{r+1-j}\overline{B}_{j}\left(
x\right)  \overline{B}_{r+1-j}\left(  y\right)  +r\overline{B}_{r+1}\left(
dy+cx\right)  .\label{recip}%
\end{align}

\begin{theorem}
Let $c$ and $d$ be coprime positive integers. For $\alpha,\beta\in\mathbb{Z}$
and $r=0,1,2,...$ the following reciprocity formula holds:%
\begin{align*}
&  dc^{r}s_{r}\left(  d,c;A_{-\alpha};B_{\beta}\right)  +cd^{r}s_{r}\left(
c,d;B_{-\beta};A_{\alpha}\right) \\
&  =\sum\limits_{j=0}^{r+1}c^{j}d^{r+1-j}P_{r+1-j}\left(  0,A_{\alpha}\right)
P_{j}\left(  0,B_{\beta}\right)  +rk^{r-1}\sum\limits_{v=1}^{k}\sum
\limits_{\mu=1}^{k}f(-\alpha v)f^{\ast}(-\beta\mu)P_{r+1}\left(  \frac
{dv+c\mu}{k}\right)  .
\end{align*}

\end{theorem}

\begin{proof}
From (\ref{1}), we have
\begin{align*}
s_{r}\left(  d,c;A_{-\alpha};B_{\beta}\right)   &  =\frac{k^{r-1}}{r!}%
\sum\limits_{v=1}^{k}\sum\limits_{\mu=1}^{k}f(-\alpha v)f^{\ast}(-\beta
\mu)s_{r}\left(  d,c|\frac{\mu}{k},\frac{v}{k}\right)  ,\\
s_{r}\left(  c,d;B_{-\beta};A_{\alpha}\right)   &  =\frac{k^{r-1}}{r!}%
\sum\limits_{v=1}^{k}\sum\limits_{\mu=1}^{k}f(-\alpha v)f^{\ast}(-\beta
\mu)s_{r}\left(  c,d|\frac{v}{k},\frac{\mu}{k}\right)
\end{align*}
and thus
\begin{align*}
&  \left(  r+1\right)  \left[  dc^{r}s_{r}\left(  d,c;A_{-\alpha};B_{\beta
}\right)  +cd^{r}s_{r}\left(  c,d;B_{-\beta};A_{\alpha}\right)  \right] \\
&  =\frac{k^{r-1}}{r!}\sum\limits_{v=1}^{k}\sum\limits_{\mu=1}^{k}f(-\alpha
v)f^{\ast}(-\beta\mu)\left(  r+1\right)  \left[  dc^{r}s_{r}\left(
d,c|\frac{\mu}{k},\frac{v}{k}\right)  +cd^{r}s_{r}\left(  c,d|\frac{v}%
{k},\frac{\mu}{k}\right)  \right]  .
\end{align*}
Using the reciprocity formula given by (\ref{recip}) and then (\ref{pr**}),
the desired result follows.
\end{proof}

\section{Values of periodic zeta-functions}

This section is devoted to derive periodic analogues of Ramanujan's formula
for $\zeta\left(  2N+1\right)  .$ We accomplish this by applying Theorem
\ref{mainteo} and (\ref{11}), motivated by \cite{7,8}.

For $z\in\mathbb{H}$ and $s$ complex, we recall the special cases of
(\ref{25}) as%
\begin{align*}
A(z,s;A_{\alpha},I;r_{1},r_{2})  &  =\sum\limits_{m>-r_{1}}f(\alpha
m)\sum\limits_{n=1}^{\infty}e\left(  n\frac{(m+r_{1})z+r_{2}}{k}\right)
n^{s-1},\\
A(z,s;I,A_{\beta};r_{1},r_{2})  &  =\sum\limits_{m>-r_{1}}\sum\limits_{n=1}%
^{\infty}f(\beta n)e\left(  n\frac{(m+r_{1})z+r_{2}}{k}\right)  n^{s-1}.
\end{align*}
Let%
\begin{equation}
H\left(  z,s;A_{\alpha},B_{\beta};r_{1},r_{2}\right)  =A\left(  z,s;A_{\alpha
},B_{-\beta};r_{1},r_{2}\right)  +e(s/2)A\left(  z,s;A_{-\alpha},B_{\beta
};-r_{1},-r_{2}\right)  \label{7}%
\end{equation}
and%
\begin{equation}
L_{\pm}\left(  s;A_{\beta};\theta\right)  =L(s;A_{\beta};\theta)+e(\pm
s/2)L(s;A_{-\beta};-\theta) \label{3}%
\end{equation}
where $L(s;A_{\beta};\theta)$ is given by (\ref{12}). In particular, let%
\begin{equation}
Z\left(  s,\theta\right)  =L(s;I;\theta)=\sum\limits_{n>-\theta}\left(
n+\theta\right)  ^{-s},\text{ for }\operatorname{Re}\left(  s\right)  >1\text{
and }\theta\text{ real} \label{6}%
\end{equation}
and%

\[
H\left(  z,s;A,B;r_{1},r_{2}\right)  =H\left(  z,s;A_{1},B_{1};r_{1}%
,r_{2}\right)  \text{ and }A\left(  z,s;A_{\alpha},B_{-\beta}\right)
=A\left(  z,s;A_{\alpha},B_{-\beta};0,0\right)  .
\]

For non-negative integers $j$ and $\mu$ and for $z\in\mathbb{K}$, define%
\begin{align}
&  I^{\ast}(z,s,c,d,r_{1},r_{2})\nonumber\\
&  \ =\int\limits_{C}u^{s-1}\frac{\exp(-\left(  (c\mu+j-\left\{
R_{1}\right\}  )/ck\right)  (cz+d)ku)}{\exp(-ku(cz+d))-1}\frac{\exp(\left\{
(dj+\rho)/c\right\}  u)}{\exp(u)-1}du. \label{int*}%
\end{align}
If $s=-N,$ where $N$ is a non-negative integer, (\ref{int}) and (\ref{int*})
can be calculated by residue theorem%
\begin{align}
&  I(z,-N,c,d,r_{1},r_{2})\label{int1}\\
&  \ =2\pi ik^{N}\sum\limits_{m+n=N+2}B_{m}\left(  \frac{c\mu+j-\left\{
R_{1}\right\}  }{ck}\right)  B_{n}\left(  \frac{v+\left\{  (dj+\rho
)/c\right\}  }{k}\right)  \frac{\left(  -(cz+d)\right)  ^{m-1}}{m!n!}\nonumber
\end{align}
and%
\begin{align}
&  I^{\ast}(z,-N,c,d,r_{1},r_{2})\label{int2}\\
&  \ =2\pi ik^{N}\sum\limits_{m+n=N+2}B_{m}\left(  \frac{c\mu+j-\left\{
R_{1}\right\}  }{ck}\right)  B_{n}\left(  \left\{  (dj+\rho)/c\right\}
\right)  k^{m-1}\frac{\left(  -(cz+d)\right)  ^{m-1}}{m!n!}\nonumber
\end{align}
respectively.

Now we state the transformation formulas involving $H\left(  z,s;A,I;r_{1}%
,r_{2}\right)  $ and $H\left(  z,s;I,A;r_{1},r_{2}\right)  .$

\begin{theorem}
\label{tf}Let $z\in\mathbb{K}$ and suppose that $s$ is an arbitrary complex
number. If $a\equiv0\left(  \operatorname{mod}k\right)  ,$ then%
\begin{align}
&  (cz+d)^{-s}\left(  -2\pi i/k\right)  ^{s}kH\left(  Vz,s;I,\widehat{B}%
;r_{1},r_{2}\right)  +\lambda_{r_{1}}(cz+d)^{-s}\Gamma(s)L_{+}\left(
s;B;r_{2}\right) \label{15}\\
&  =\left(  -2\pi i/k\right)  ^{s}H\left(  z,s;B_{-b},I;R_{1},R_{2}\right)
+\lambda_{R_{1}}f^{\ast}(bR_{1})\Gamma(s)Z_{-}\left(  s,R_{2}\right)
\nonumber\\
&  +e(-s/2)\sum\limits_{j=1}^{c}\sum\limits_{\mu=0}^{k-1}f^{\ast}%
(b(c\mu+j+\left[  R_{1}\right]  ))I^{\ast}(z,s,c,d,r_{1},r_{2}).\nonumber
\end{align}
If $b\equiv0\left(  \operatorname{mod}k\right)  ,$ then%
\begin{align}
&  (cz+d)^{-s}\left(  -2\pi i/k\right)  ^{s}kH\left(  Vz,s;I,\widehat{B}%
;r_{1},r_{2}\right)  +\lambda_{r_{1}}(cz+d)^{-s}\Gamma(s)L_{+}\left(
s,B,r_{2}\right) \label{16}\\
&  =\left(  -2\pi i/k\right)  ^{s}kH\left(  z,s;I,\widehat{B}_{a};R_{1}%
,R_{2}\right)  +\lambda_{R_{1}}\Gamma(s)L_{-}\left(  s;B_{a};R_{2}\right)
\nonumber\\
&  +e(-s/2)\sum\limits_{j=1}^{c}\sum\limits_{\mu=0}^{k-1}\sum\limits_{v=0}%
^{k-1}f^{\ast}(-a(\left[  R_{2}+d(j-\left\{  R_{1}\right\}  )/c\right]
-v+d\mu))I(z,s,c,d,r_{1},r_{2}).\nonumber
\end{align}
If $c\equiv0\left(  \operatorname{mod}k\right)  ,$ then%
\begin{align}
&  (cz+d)^{-s}\left(  -2\pi i/k\right)  ^{s}H\left(  Vz,s;A,I;r_{1}%
,r_{2}\right)  +\lambda_{r_{1}}f\left(  -r_{1}\right)  (cz+d)^{-s}%
\Gamma(s)Z_{+}\left(  s,r_{2}\right) \label{17}\\
&  =\left(  -2\pi i/k\right)  ^{s}H\left(  z,s;A_{d},I;R_{1},R_{2}\right)
+\lambda_{R_{1}}f(-dR_{1})\Gamma(s)Z_{-}\left(  s,R_{2}\right) \nonumber\\
&  +e(-s/2)\sum\limits_{j=1}^{c}\sum\limits_{\mu=0}^{k-1}f(-d(c\mu+j+\left[
R_{1}\right]  ))I^{\ast}(z,s,c,d,r_{1},r_{2}).\nonumber
\end{align}
If $d\equiv0\left(  \operatorname{mod}k\right)  ,$ then%
\begin{align}
&  (cz+d)^{-s}\left(  -2\pi i/k\right)  ^{s}H\left(  Vz,s;A,I;r_{1}%
,r_{2}\right)  +\lambda_{r_{1}}f\left(  -r_{1}\right)  (cz+d)^{-s}%
\Gamma(s)Z_{+}\left(  s,r_{2}\right) \label{18}\\
&  =\left(  -2\pi i/k\right)  ^{s}kH\left(  z,s;I,\widehat{A}_{-c};R_{1}%
,R_{2}\right)  +\lambda_{R_{1}}\Gamma(s)L_{-}\left(  s;A_{-c};R_{2}\right)
\nonumber\\
&  +e(-s/2)\sum\limits_{j=1}^{c}\sum\limits_{\mu=0}^{k-1}\sum\limits_{v=0}%
^{k-1}f(c(\left[  R_{2}+d(j-\left\{  R_{1}\right\}  )/c\right]
-v))I(z,s,c,d,r_{1},r_{2}).\nonumber
\end{align}
Furthermore, if $s=-N$ is non-negative integer, upon the evaluation of
$I(z,s,c,d,r_{1},r_{2})$ and $I^{\ast}(z,s,c,d,r_{1},r_{2})$ by (\ref{int1})
and (\ref{int2}), respectively, (\ref{15})--(\ref{18}) are valid for
$z\in\mathbb{H}$.
\end{theorem}

\begin{proof}
For $z\in\mathbb{K}$, $\operatorname{Re}\left(  s\right)  >2,$ $M=ma+nc$ and
$N=mb+nd,$ we have
\begin{align*}
&  G(Vz,s;I,B;r_{1},r_{2})\\
&  =\sum\limits_{M,N=-\infty}^{\infty}f^{\ast}(Na-Mb)\left\{  \frac
{((M+R_{1})z+N+R_{2})}{cz+d}\right\}  ^{-s}\\
&  =\sum\limits_{m,n=-\infty}^{\infty}f^{\ast}(-mb)\left\{  \frac
{((m+R_{1})z+n+R_{2})}{cz+d}\right\}  ^{-s},\text{ }a\equiv
0(\operatorname{mod}k),\\
&  =\sum\limits_{m,n=-\infty}^{\infty}f^{\ast}(an)\left\{  \frac
{((m+R_{1})z+n+R_{2})}{cz+d}\right\}  ^{-s},\text{ }b\equiv
0(\operatorname{mod}k).
\end{align*}
To prove (\ref{15}) and (\ref{16}), we follow precisely the method in the
proof of \cite[Theorem 1]{dc} then use Lemma \ref{lem} and Eqs. (\ref{7}%
)--(\ref{6}).\ To prove (\ref{17}) and (\ref{18}), follow the method outlined
above, but begin by examining $G(Vz,s;A,I;r_{1},r_{2})$ instead of
$G(Vz,s;I,B;r_{1},r_{2}).$
\end{proof}

We mention that the character versions of this theorem are given by Berndt
\cite[Theorem 4.1]{8} for primitive characters $f=\chi_{1}$ and $f^{\ast}%
=\chi_{2},$ and in \cite[Theorem 3]{7} for $\chi_{1}=\chi=\overline{\chi}%
_{2}.$

By letting $s=-N$ be a non-positive integer in Theorem 3 of \cite{7}, Berndt
obtain some interesting formulas for Dirichlet\textbf{ }$L$--functions or
curious arithmetical results that are the character analogues of Ramanujan's
formula for $\zeta\left(  2N+1\right)  $ via transformation formulas. Periodic
analogues of these formulas are due to Bradley \cite{brad} and are special
cases of Theorem \ref{tf}.

\begin{theorem}
\label{rf}(Bradley \cite{brad}) Let $N$ be positive integer and let $\gamma$
be an arbitrary positive number.

If $f$ is an even function, then
\begin{align}
&  \zeta\left(  2N+1,\widehat{A}\right)  -f\left(  0\right)  (i\gamma
k^{2})^{-2N}\zeta\left(  2N+1\right) \nonumber\\
&  =2\left(  i\gamma k\right)  ^{-2N}A\left(  ik\gamma,-2N;A,I\right)
-2kA\left(  \frac{i}{k\gamma},-2N;I,\widehat{A}\right) \nonumber\\
&  \quad+\frac{\left(  2\pi i\right)  ^{2N+1}}{k^{2N}}\sum\limits_{m=0}%
^{N+1}P_{2m}\left(  0\right)  P_{2N+2-2m}\left(  0,A\right)  \left(
i/k\gamma\right)  ^{2m-1}, \label{21}%
\end{align}
where $\zeta\left(  s,A\right)  =L\left(  s;A;0\right)  $ is the periodic
zeta-function \cite[Sec. 6]{10}. If $f\left(  0\right)  =0,$ (\ref{21}) is
also valid for $N=0$.

If $f$ is an odd function, then
\begin{align}
\zeta\left(  2N,\widehat{A}\right)   &  =-2\left(  ik\gamma\right)
^{1-2N}A\left(  ik\gamma,1-2N;A,I\right)  +2kA\left(  \frac{i}{k\gamma
},1-2N;I,\widehat{A}\right) \nonumber\\
&  \quad+\frac{\left(  2\pi i\right)  ^{2N}}{k^{2N-1}}\sum\limits_{m=0}%
^{N}P_{2m}\left(  0\right)  P_{2N+1-2m}\left(  0,A\right)  \left(
i/k\gamma\right)  ^{2m-1}. \label{22}%
\end{align}

\end{theorem}

\begin{proof}
Using the functional equation of $\zeta(s,A)$ (\cite[Corollary 6.5]{10}), we
have%
\begin{equation}
\lim_{s\rightarrow-N}\Gamma(s)L_{-}\left(  s;A_{c};0\right)  =e^{-\pi
iN/2}\left(  k/2\pi\right)  ^{N}\zeta\left(  N+1,\widehat{A}_{c}\right)  .
\label{20}%
\end{equation}
Suppose that $f$ is even. Using (\ref{int1}) and (\ref{7}), equation
(\ref{18}) can be written as
\begin{align*}
&  (cz+d)^{N}\left(  -2\pi i/k\right)  ^{-N}\left(  1+e\left(  -N/2\right)
\right)  A\left(  Vz,-N;A,I\right)  +f\left(  0\right)  \lim_{s\rightarrow
-N}(cz+d)^{-s}\Gamma(s)Z_{+}\left(  s,0\right) \\
&  =\left(  -2\pi i/k\right)  ^{-N}k\left(  1+e\left(  -N/2\right)  \right)
A\left(  z,-N;I,\widehat{A}_{-c}\right)  +\lim_{s\rightarrow-N}\Gamma
(s)L_{-}\left(  s;A_{-c};0\right) \\
&  +2\pi i\frac{\left(  -k\right)  ^{N}}{\left(  N+2\right)  !}\sum
\limits_{j=1}^{c}\sum\limits_{\mu=0}^{k-1}\sum\limits_{v=0}^{k-1}%
\sum\limits_{m=0}^{N+2}\binom{N+2}{m}f(c(\left[  dj/c\right]  -v))\\
&  \quad\times B_{m}\left(  \frac{c\mu+j}{ck}\right)  B_{N+2-m}\left(
\frac{v+\left\{  dj/c\right\}  }{k}\right)  \left(  -(cz+d)\right)  ^{m-1}.
\end{align*}
Replace $N$ by $2N,$ let $Vz=-1/z$ and put $z=i/k\gamma,$ $\gamma>0$. From the
facts (\ref{10-1}), (\ref{pr**}) and (\ref{20}), we have%
\begin{align}
&  2\left(  2\pi\gamma\right)  ^{-2N}A\left(  ik\gamma,-2N;A,I\right)
+f\left(  0\right)  (2\pi k\gamma)^{-2N}\zeta\left(  2N+1\right) \nonumber\\
&  =2\left(  2\pi i/k\right)  ^{-2N}kA\left(  \frac{i}{k\gamma}%
,-2N;I,\widehat{A}\right)  +\left(  k/2\pi\right)  ^{2N}\left(  -1\right)
^{N}\zeta\left(  2N+1,\widehat{A}_{-1}\right) \nonumber\\
&  +2\pi i\sum\limits_{m=0}^{2N+2}\left(  -1\right)  ^{m}P_{m}\left(
0\right)  P_{2N+2-m}\left(  0,A\right)  \left(  -i/k\gamma\right)  ^{m-1},
\label{27}%
\end{align}
where we have used that
\[
\sum\limits_{\mu=0}^{k-1}B_{m}\left(  \frac{\mu+1}{k}\right)  =k^{1-m}\left(
-1\right)  ^{m}B_{m}\left(  0\right)  .
\]
Since $P_{2N+2-m}\left(  0,A\right)  P_{m}\left(  0\right)  =0$ for odd $m$
and even $f,$ we can replace $m$ by $2m$ on the right-hand side of (\ref{27}),
which completes the proof of (\ref{21}).

Similarly, if $f$ is an odd function, then replacing $N$ by $2N-1$ and
proceeding as in the proof of (\ref{21}) give%
\begin{align}
&  2\left(  ik\gamma\right)  ^{1-2N}A\left(  ik\gamma,1-2N;A,I\right)
-2kA\left(  \frac{i}{k\gamma},1-2N;I,\widehat{A}\right) \nonumber\\
&  =\zeta\left(  2N,\widehat{A}_{-1}\right)  +\frac{\left(  2\pi i\right)
^{2N}}{k^{2N-1}}\sum\limits_{m=0}^{2N+1}P_{2N+1-m}\left(  0,A\right)
P_{m}\left(  0\right)  \left(  i/k\gamma\right)  ^{m-1}, \label{32}%
\end{align}
which completes the proof of (\ref{22}) by replacing $m$ by $2m$ on the
right-hand side of (\ref{32}).
\end{proof}

Note that Theorem \ref{rf} is just one of an infinite class of such formulas
that can be deduced from Theorem \ref{tf} when $s=-N$ and $r_{1}=r_{2}=0.$
Similar formulas for $\zeta\left(  2N+1,B\right)  $ and $\zeta\left(
2N,B\right)  $ can be obtained from (\ref{15}). Moreover, using the following
relations, (\ref{21}) and (\ref{22}) can be written in terms of
zeta-functions:
\[
B_{2r}=\frac{2\left(  -1\right)  ^{r-1}\left(  2r\right)  !}{\left(
2\pi\right)  ^{2r}}\zeta\left(  2r\right)
\]
and \cite[Eqs. (6.23) and (6.25)]{10}
\begin{equation}
\zeta\left(  r,\widehat{A}\right)  =-\dfrac{1}{2}\dfrac{\left(  2\pi
i/k\right)  ^{r}}{r!}B_{r}\left(  0,A\right)  ,\text{ }r\geq1 \label{28}%
\end{equation}
when $r$ and $f$ have the same parity.

We also note that from Theorem \ref{tf}, the values of $\zeta\left(
N+1,A\right)  $ can be deduced when $N$ and $f$ have the opposite parity. For
$r_{1}=r_{2}=0$ and $s=-N$ in (\ref{7}), we have%
\begin{align*}
H\left(  Vz,-N;A,I;0,0\right)   &  =A\left(  Vz,-N;A,I\right)
+e(-N/2)A\left(  Vz,-N;A_{-1},I\right)  =0,\\
H\left(  z,s;I,\widehat{A}_{-c};0,0\right)   &  =A\left(  z,-N;I,\widehat{A}%
_{c}\right)  +e(-N/2)A\left(  z,-N;I,\widehat{A}_{-c}\right)  =0.
\end{align*}
Thus, using (\ref{int1}), (\ref{pr**}) and (\ref{20}), equation (\ref{18})
turns into%
\[
\zeta\left(  N+1,\widehat{A}_{-1}\right)  -f\left(  0\right)  z^{N}%
\zeta\left(  N+1\right)  =\frac{\left(  2\pi i\right)  ^{N+1}}{\left(
-k\right)  ^{N}}\sum\limits_{m=0}^{N+2}P_{N+2-m}\left(  0,A\right)
P_{m}\left(  0\right)  z^{m-1}%
\]
when $N$ and $f$ have the opposite parity.

For some special cases of $r_{1}$ and $r_{2}$ in Theorem \ref{tf}, we may
achieve the following formulas.

\begin{theorem}
\label{rfH}Let $N$ denote a non-negative integer and $0<R<1.$ Then,%
\begin{align}
&  \left(  -i\gamma k\right)  ^{-N}H\left(  ik\gamma,-N;A,I;-R,0\right)
-kH\left(  \frac{i}{k\gamma},-N;I,\widehat{A}_{-1};0,R\right) \label{19}\\
&  =\varphi\left(  R,0,1+N;\widehat{A}_{-1}\right)  -\frac{\left(  2\pi
i\right)  ^{N+1}}{k^{N}}\sum\limits_{m=0}^{N+2}P_{m}\left(  0\right)
P_{N+2-m}\left(  R,A\right)  \left(  i/k\gamma\right)  ^{m-1}\nonumber
\end{align}
and%
\begin{align}
&  kH\left(  ik\gamma,-N;I,\widehat{B};0,R\right)  +\left(  -1\right)
^{N}k\varphi\left(  -R,0,N+1;\widehat{B}\right) \label{19a}\\
&  =\left(  -ik\gamma\right)  ^{N}H\left(  \frac{i}{k\gamma}%
,-N;B,I;R,0\right)  -\left(  2\pi i\right)  ^{N+1}\sum\limits_{m=0}^{N+2}%
P_{m}\left(  R,B_{-1}\right)  P_{N+2-m}\left(  0\right)  \left(
i/k\gamma\right)  ^{m-N-1},\nonumber
\end{align}
where $\varphi\left(  x,a,s;A\right)  $ is the periodic Lerch function defined
by \cite{10}
\[
\varphi\left(  x,a,s;A\right)  =\sum\limits_{n=0}^{\infty}\ \hspace
{-0.05in}^{^{\prime}}f\left(  n\right)  e^{2\pi inx/k}\left(  n+a\right)
^{-s},\text{ for }\operatorname{Re}\left(  s\right)  >1\text{ and }x,a\text{
real}%
\]
in which the dash $\prime$ indicates that if $a$ is a non-positive integer,
then the term corresponding to $n=-a$ is omitted from the sum.
\end{theorem}

\begin{proof}
Let $Vz=-1/z$ and put $z=i/k\gamma,$ $\gamma>0$. Then, $R_{1}=r_{2}$ and
$R_{2}=-r_{1}.$ For the proof of (\ref{19}), put $r_{2}=0$ and $R:=R_{2}$ and
suppose that $0<R<1$ in (\ref{18}). Hence, by aid of (\ref{int1}), we write
\begin{align*}
&  \left(  \frac{i}{k\gamma}\right)  ^{N}\left(  -2\pi i/k\right)
^{-N}H\left(  ik\gamma,-N;A,I;-R,0\right) \\
&  =\left(  -2\pi i/k\right)  ^{-N}kH\left(  \frac{i}{k\gamma}%
,-N;I,\widehat{A}_{-1};0,R\right)  +\lim_{s\rightarrow-N}\Gamma(s)L_{-}\left(
s;A_{-1};R\right) \\
&  +e(N/2)2\pi ik^{N}\sum\limits_{\mu=0}^{k-1}\sum\limits_{v=0}^{k-1}%
\sum\limits_{m=0}^{N+2}f(-v))B_{m}\left(  \frac{\mu+1}{k}\right)
B_{N+2-m}\left(  \frac{v+R}{k}\right)  \frac{\left(  -i/k\gamma\right)
^{m-1}}{m!\left(  N+2-m\right)  !}.
\end{align*}
Using (\ref{26}) and (\ref{10-1}), we have
\begin{align}
&  \left(  -2\pi\gamma\right)  ^{-N}H\left(  ik\gamma,-N;A,I;-R,0\right)
\nonumber\\
&  =\left(  -2\pi i/k\right)  ^{-N}kH\left(  \frac{i}{k\gamma}%
,-N;I,\widehat{A}_{-1};0,R\right)  +\lim_{s\rightarrow-N}\Gamma(s)L_{-}\left(
s;A_{-1};R\right) \label{33}\\
&  +e(N/2)2\pi i\sum\limits_{m=0}^{N+2}\left(  -1\right)  ^{m}P_{m}\left(
0\right)  P_{2N+2-m}\left(  R,A\right)  \left(  -i/k\gamma\right)
^{m-1}.\nonumber
\end{align}
We must evaluate the limit above. For $0<R<1,$ we have%
\begin{align}
L_{-}\left(  s;A_{-1};R\right)   &  =L(s;A_{-1};R)+e^{-\pi is}%
L(s;A;-R)\nonumber\\
&  =e^{-\pi is/2}\left\{  e^{\pi is/2}\sum\limits_{n=0}^{\infty}\frac
{f(-n)}{\left(  n+R\right)  ^{s}}+e^{-\pi is/2}\sum\limits_{n=1}^{\infty}%
\frac{f(n)}{\left(  n-R\right)  ^{s}}\right\} \nonumber\\
&  =e^{-\pi is/2}\left\{  e^{\pi is/2}\sum\limits_{n=0}^{\infty}\frac
{f(-n)}{\left(  n+R\right)  ^{s}}+e^{-\pi is/2}\sum\limits_{n=0}^{\infty}%
\frac{f(n+1)}{\left(  n+1-R\right)  ^{s}}\right\}  . \label{49}%
\end{align}
Now we make use of the functional equation of the periodic Lerch function
\cite[Theorem 6.1]{10}
\begin{align}
&  \varphi\left(  x,a,1-s;\widehat{C}\right)  \text{ }\nonumber\\
&  \ =\left(  k/2\pi\right)  ^{s}\Gamma(s)\left\{  e^{\pi is/2-2\pi
iax/k}\varphi\left(  -a,x,s;C\right)  +e^{-\pi is/2+2\pi ia(1-x)/k}%
\varphi\left(  a,1-x,s;C^{\ast}\right)  \right\}  , \label{41}%
\end{align}
where $C^{\ast}=\left\{  g\left(  -n-1\right)  \right\}  $ for $C=\left\{
g\left(  n\right)  \right\}  .$ So, taking $a=0,$ $x=R,$ $C=A_{-1}=\left\{
f(-n)\right\}  $ and $C^{\ast}=\left\{  f\left(  n+1\right)  \right\}  $ in
(\ref{41}), (\ref{49}) can be written as%
\[
\Gamma(s)L_{-}\left(  s;A_{-1};R\right)  =e^{-\pi is/2}\Gamma(s)\left\{
e^{\pi is/2}\varphi\left(  0,R,s;A_{-1}\right)  +e^{-\pi is/2}\varphi\left(
0,1-R,s;C^{\ast}\right)  \right\}
\]
which implies
\[
\lim_{s\rightarrow-N}\Gamma(s)L_{-}\left(  s;A_{-1};R\right)  =\left(
2\pi/k\right)  ^{-N}e^{\pi iN/2}\varphi\left(  R,0,1+N;\widehat{A}%
_{-1}\right)  .
\]
Substituting this in (\ref{33}), the desired result follows.

For the proof of (\ref{19a}), we put $R_{2}=r_{1}=0$ and $0<R:=R_{1}=r_{2}<1$
in (\ref{15}). Firstly, we evaluate the limit $\lim_{s\rightarrow-N}%
\Gamma(s)L_{+}\left(  s;B;R\right)  .$

Using (\ref{3}) and replacing $n\rightarrow nk+v$, we have%
\begin{align*}
\Gamma(s)L_{+}\left(  s;B;R\right)   &  =\Gamma(s)\left\{  L(s;B;R)+e^{\pi
is}L(s;B_{-1};-R)\right\} \\
&  =\frac{1}{k^{s}}\sum\limits_{v=0}^{k-1}f^{\ast}(v)\Gamma(s)\left\{
\zeta\left(  s,\frac{v+R}{k}\right)  +e^{\pi is}\zeta\left(  s,1-\frac{v+R}%
{k}\right)  \right\}  .
\end{align*}
Now, by Hurwitz's formula for $\zeta\left(  s,x\right)  $ for
$\operatorname{Re}s<0$ \cite[p. 269]{ww} (or see \cite[Eq. 3.5]{8} for detail)
the right hand-side becomes%
\begin{align*}
&  \frac{\left(  2\pi i\right)  ^{s}}{k^{s}}\sum\limits_{v=0}^{k-1}f^{\ast
}(v)\varphi\left(  -\frac{v+R}{k},1-s\right) \\
&  =\frac{\left(  2\pi i\right)  ^{s}}{k^{s}}\sum\limits_{n=1}^{\infty}%
\frac{e^{-2\pi inR/k}}{n^{1-s}}\sum\limits_{v=0}^{k-1}f^{\ast}(v)e^{-2\pi
inv/k}\\
&  =\frac{\left(  2\pi i\right)  ^{s}}{k^{s-1}}\sum\limits_{n=1}^{\infty
}\widehat{f^{\ast}}(n)\frac{e^{-2\pi inR/k}}{n^{1-s}}=\frac{\left(  2\pi
i\right)  ^{s}}{k^{s-1}}\varphi\left(  -R,0,1-s;\widehat{B}\right)  .
\end{align*}
Namely,%
\begin{equation}
\lim_{s\rightarrow-N}\Gamma(s)L_{+}\left(  s;B;R\right)  =\frac{\left(  2\pi
i\right)  ^{-N}}{k^{-N-1}}\varphi\left(  -R,0,N+1;\widehat{B}\right)  .
\label{38}%
\end{equation}
Using (\ref{int2}), we write%
\begin{align}
&  \sum\limits_{\mu=0}^{k-1}f^{\ast}(b(\mu+1+\left[  R_{1}\right]  ))I^{\ast
}(z,s,1,0,0,R)\nonumber\\
&  =2\pi ik^{N}\sum\limits_{m=0}^{N+2}B_{N+2-m}\left(  0\right)  k^{m-1}%
\frac{(-z)^{m-1}}{m!\left(  N+2-m\right)  !}\sum\limits_{\mu=0}^{k-1}f^{\ast
}(-(\mu+1))B_{m}\left(  \frac{\mu+1-R}{k}\right) \nonumber\\
&  =2\pi ik^{N}\sum\limits_{m=0}^{N+2}\left(  -1\right)  ^{m}P_{m}\left(
R,B_{-1}\right)  P_{N+2-m}\left(  0\right)  (-z)^{m-1} \label{39}%
\end{align}
where we have used that $B_{m}\left(  1-x\right)  =\left(  -1\right)
^{m}B_{m}\left(  x\right)  $ and (\ref{pr**}). Substituting (\ref{38}) and
(\ref{39}) in (\ref{15}) and simplifying give (\ref{19a}).
\end{proof}

Note that if we take $A=I,$ the statement given by (\ref{19a}) in Theorem
\ref{rfH} reduces to \cite[Theorem 3.1]{8}. Also Theorem 3.3 of \cite{8} can
be obtained from (\ref{19}) by observing that
\begin{align*}
H\left(  z,-N;I,I;-R,0\right)   &  =A\left(  z,-N;I,I;-R,0\right)
+e(-N/2)A\left(  z,-N;I,I;R,0\right) \\
&  =\sum\limits_{n=1}^{\infty}\frac{n^{-N-1}}{e^{-2\pi inz}-1}\left(  e^{-2\pi
inzR}+\left(  -1\right)  ^{N}e^{2\pi inzR}\right)  +\sum\limits_{n=1}^{\infty
}\frac{e^{2\pi inzR}}{n^{N+1}},\\
H\left(  z,-N;I,I;0,R\right)   &  =\sum\limits_{n=1}^{\infty}\frac{n^{-N-1}%
}{e^{-2\pi inz}-1}\left(  e^{2\pi inR}+\left(  -1\right)  ^{N}e^{-2\pi
inR}\right)  .
\end{align*}

Now, we assume that $f$ is odd or even (and of course $\widehat{f}$ ). For
convenience, put $f\left(  -m\right)  =\delta f\left(  m\right)  $, where
$\delta=1$ if $f$ is even and $\delta=-1$ if $f$ is odd. From (\ref{7}), with
(\ref{2}), we have%
\begin{align*}
&  H\left(  z,-N;A,I;-R,0\right) \\
&  =A\left(  z,-N;A,I;-R,0\right)  +\left(  -1\right)  ^{N}A\left(
z,-N;A_{-1},I;R,0\right) \\
&  =\sum\limits_{n=1}^{\infty}n^{-N-1}\sum\limits_{m=1}^{\infty}f\left(
m\right)  e^{2\pi inmz/k}\left(  e^{-2\pi inRz/k}+\left(  -1\right)
^{N}\delta e^{2\pi inRz/k}\right)  +f\left(  0\right)  \left(  -1\right)
^{N}\sum\limits_{n=1}^{\infty}n^{-N-1}e^{2\pi inRz/k}\\
&  =\sum\limits_{n=1}^{\infty}\sum\limits_{v=0}^{k-1}\frac{\widehat{f}\left(
v\right)  n^{-N-1}}{e^{-2\pi i\left(  v+nz\right)  /k}-1}\left(  e^{-2\pi
inRz/k}+\left(  -1\right)  ^{N}\delta e^{2\pi inRz/k}\right)  +f\left(
0\right)  \left(  -1\right)  ^{N}\sum\limits_{n=1}^{\infty}\frac{e^{2\pi
inRz/k}}{n^{N+1}}%
\end{align*}
and
\begin{align*}
H\left(  z,-N;I,\widehat{A}_{-1};0,R\right)   &  =\sum\limits_{n=1}^{\infty
}\widehat{f}\left(  n\right)  n^{-N-1}\sum\limits_{m=1}^{\infty}e^{2\pi
inmz/k}\left(  e^{2\pi inR/k}+\left(  -1\right)  ^{N}\delta e^{-2\pi
inR/k}\right) \\
&  =\sum\limits_{n=1}^{\infty}\widehat{f}\left(  n\right)  \frac{n^{-N-1}%
}{e^{-2\pi inz/k}-1}\left(  e^{2\pi inR/k}+\left(  -1\right)  ^{N}\delta
e^{-2\pi inR/k}\right)  .
\end{align*}
Thus, substituting these in (\ref{19}) we have
\begin{align}
&  2\left(  -i\gamma k\right)  ^{-N}\sum\limits_{n=1}^{\infty}\sum
\limits_{v=0}^{k-1}\widehat{f}\left(  v\right)  \frac{\cosh\left(  2\pi
nR\gamma\right)  }{e^{-2\pi i\left(  v+nik\gamma\right)  /k}-1}n^{-N-1}%
+f\left(  0\right)  \left(  i\gamma k\right)  ^{-N}\sum\limits_{n=1}^{\infty
}\frac{e^{-2\pi nR\gamma}}{n^{N+1}}\nonumber\\
&  \ =\delta\sum\limits_{n=1}^{\infty}\widehat{f}\left(  n\right)
\frac{e^{2\pi inR/k}}{n^{N+1}}-2k\sum\limits_{n=1}^{\infty}\widehat{f}\left(
n\right)  \frac{\cos\left(  2\pi nR/k\right)  }{e^{2\pi n/k^{2}\gamma}%
-1}n^{-N-1}\nonumber\\
&  \quad-\frac{\left(  2\pi i\right)  ^{N+1}}{k^{N}}\sum\limits_{m=0}%
^{N+2}P_{m}\left(  0\right)  P_{N+2-m}\left(  R,A\right)  \left(
i/k\gamma\right)  ^{m-1} \label{42}%
\end{align}
if $\left(  -1\right)  ^{N}\delta=1$ and
\begin{align}
&  2\left(  -i\gamma k\right)  ^{-N}\sum\limits_{n=1}^{\infty}\sum
\limits_{v=0}^{k-1}\widehat{f}\left(  v\right)  \frac{\sinh\left(  2\pi
nR\gamma\right)  }{e^{-2\pi i\left(  v+nz\right)  /k}-1}n^{-N-1}+f\left(
0\right)  \left(  i\gamma k\right)  ^{-N}\sum\limits_{n=1}^{\infty}%
\frac{e^{-2\pi nR\gamma}}{n^{N+1}}\nonumber\\
&  \ =\delta\sum\limits_{n=1}^{\infty}\widehat{f}\left(  n\right)
\frac{e^{2\pi inR/k}}{n^{N+1}}-2ik\sum\limits_{n=1}^{\infty}\widehat{f}\left(
n\right)  \frac{\sin\left(  2\pi nR/k\right)  }{e^{2\pi n/k^{2}\gamma}%
-1}n^{-N-1}\nonumber\\
&  \quad-\frac{\left(  2\pi i\right)  ^{N+1}}{k^{N}}\sum\limits_{m=0}%
^{N+2}P_{m}\left(  0\right)  P_{N+2-m}\left(  R,A\right)  \left(
i/k\gamma\right)  ^{m-1} \label{43}%
\end{align}
if $\left(  -1\right)  ^{N}\delta=-1$.

Formulas (\ref{42}) and (\ref{43}) can be specialized as in Theorem \ref{rf}
and are periodic analogues of \cite[Theorem 3.3]{8}. Letting $r$ tend to $0$
in (\ref{42}) and (\ref{43}), we get Bradley's formulas (\ref{21}) and
(\ref{22}), respectively.

We conclude the study with the following remark that mention from the
relationship between periodic zeta function and Dedekind sums.

\begin{remark}
By comparing the definitions of $P_{r}\left(  0,A_{d}\right)  $ and Dedekind
sum, these are closely related for\textbf{ }$A=\left\{  f\left(  n\right)
\right\}  =\left\{  P_{q}\left(  n/k\right)  \right\}  .$ More precisely,
\[
P_{r}\left(  0,A_{d}\right)  =k^{r-1}\sum\limits_{m=0}^{k-1}P_{q}\left(
-\frac{dm}{k}\right)  P_{r}\left(  \frac{m}{k}\right)  =k^{r-1}s_{q,r}\left(
-d,k\right)  ,
\]
where $s_{q,r}\left(  d,k\right)  $ denotes the higher order (inhomogeneous)
Dedekind sum.

Combining the last identity and (\ref{28}) entails for $r\geq1$ and $q\geq2$
that%
\begin{equation}
\zeta\left(  r,\widehat{A}_{d}\right)  =\left(  -1\right)  ^{q+1}%
\dfrac{\left(  2\pi i\right)  ^{r}}{2k}s_{q,r}\left(  d,k\right)  , \label{48}%
\end{equation}
when $r$ and $q$ have the same parity.

Because of this relationship, we evaluate $\widehat{A}=\left\{  \widehat{f}%
(n)\right\}  $ when $A=\left\{  f\left(  n\right)  \right\}  =\left\{
P_{q}\left(  n/k\right)  \right\}  .$ From (\ref{47}), we write
\begin{align*}
\widehat{f}(n)  &  =\dfrac{1}{k}\sum\limits_{v=0}^{k-1}P_{q}\left(  \frac
{v}{k}\right)  e^{-2\pi inv/k}=\dfrac{1}{k^{q}}k^{q-1}\sum\limits_{v=0}%
^{k-1}e^{-2\pi inv/k}P_{q}\left(  \frac{v}{k}\right) \\
&  =\dfrac{1}{k^{q}}P_{q}\left(  0,C_{n}\right)  ,\text{ }%
\end{align*}
where $C=\left\{  g\left(  j\right)  \right\}  =\left\{  e^{2\pi
ij/k}\right\}  $ and $C_{n}=\left\{  g\left(  nj\right)  \right\}  .$ From
\cite[Corollary 6.4]{10}, we have%
\[
\zeta\left(  1-q,C_{n}\right)  =\left(  -1\right)  ^{q-1}\left(  q-1\right)
!P_{q}\left(  0,C_{n}\right)  ,\text{ }q\geq2.
\]
Since
\begin{align*}
\zeta\left(  s,C_{n}\right)   &  =\sum\limits_{m=1}^{\infty}\frac{e^{2\pi
im\frac{n}{k}}}{m^{s}}=\varphi\left(  \frac{n}{k},s\right)  ,\\
\varphi\left(  x,1-q\right)   &  =-\frac{\beta_{q}\left(  e^{2\pi ix}\right)
}{q},\text{ }q\geq1\text{ \ (\cite{1a})}%
\end{align*}
where $\varphi\left(  x,s\right)  =\varphi\left(  x,0,s;I\right)  $ is the
Lerch function and $\beta_{q}\left(  \alpha\right)  $ is the Apostol-Bernoulli
numbers, we have
\begin{align*}
\widehat{f}(n)  &  =\dfrac{1}{k^{q}}P_{q}\left(  0,C_{n}\right) \\
&  =\dfrac{\left(  -1\right)  ^{q-1}}{\left(  q-1\right)  !k^{q}}\zeta\left(
1-q,C_{n}\right)  ,\text{ }q\geq2\\
&  =\dfrac{\left(  -1\right)  ^{q-1}}{\left(  q-1\right)  !k^{q}}\phi\left(
\frac{n}{k},1-q\right)  ,\text{ }q\geq2\\
&  =\dfrac{\left(  -1\right)  ^{q}}{q!k^{q}}\beta_{q}\left(  e^{2\pi i\frac
{n}{k}}\right)  ,\text{ }q\geq2.
\end{align*}
Hence, for coprime integers $d$ and $k,$ this and (\ref{48}) give
\[
\zeta\left(  r,\widehat{A}_{d}\right)  =\dfrac{\left(  -1\right)  ^{q}%
}{q!k^{q}}\sum\limits_{m=1}^{\infty}\frac{\beta_{q}\left(  e^{2\pi i\frac
{m}{k}d}\right)  }{m^{r}}=\left(  -1\right)  ^{q+1}\dfrac{\left(  2\pi
i\right)  ^{r}}{2k}s_{q,r}\left(  d,k\right)
\]
when $r\geq1$ and $q\geq2$ have the same parity.
\end{remark}

\end{document}